\documentclass[12pt]{amsart}
\usepackage{amsfonts}
\usepackage{amsfonts, mathdots}
\usepackage{amssymb,latexsym,color}
\usepackage{enumerate}\usepackage{arydshln}
\usepackage{hyperref}
\makeatletter
\@namedef{subjclassname@2010}{%
  \textup{2010} Mathematics Subject Classification}
\makeatother

\newtheorem{thm}{Theorem}
\newtheorem{cor}[thm]{Corollary}
\newtheorem{lem}[thm]{Lemma}

\theoremstyle{definition}

\newtheorem{rem}[thm]{Remark}

\begin{document}

\baselineskip=17pt

\title[determinantal  inequality]{An extension of Harnack type determinantal inequality\footnote{To appear in Linear and Multilinear Algebra.}}

\author[M. Lin]{Minghua Lin}
\address{Department of Mathematics \\
  Shanghai University\\ Shanghai, 200444, China}
\email{m\_lin@i.shu.edu.cn}
 \author[F. Zhang]{Fuzhen Zhang}
 \address{Department of Mathematics\\
 Nova Southeastern University\\
 Fort Lauderdale, FL 33314, USA}
 \email{zhang@nova.edu}

\date{}

\begin{abstract} We revisit and comment on the Harnack type determinantal inequality for contractive matrices obtained by Tung in the sixtieth and give an extension of  the inequality involving  multiple positive semidefinite matrices.
  \end{abstract}



\maketitle

{\em In memory of Marvin Marcus.}

\bigskip

In 1964,
Tung \cite{Tun64}  established the following Harnack type determinantal  inequality.

\begin{thm}\label{thm1}
 Let $Z$ be an $n\times n$ complex matrix with singular values $r_k$  that satisfy  $0\le r_k<1$,   $k=1, 2, \ldots, n$ (i.e., $Z$ is a strict contraction). Let $Z^*$ denote the conjugate transpose of $Z$ and $I$ be the $n\times n$ identity matrix.  Then for any $n\times n$ unitary matrix \nolinebreak  $U$
\begin{eqnarray}\label{e1}
\prod_{k=1}^n\frac{1-r_k}{1+r_k}\le \frac{\det(I-Z^*Z)}{|\det(I-UZ)|^2}\le \prod_{k=1}^n\frac{1+r_k}{1-r_k}.
\end{eqnarray}
\end{thm}

 Soon after the appearance of the Tung's paper, Marcus \cite{Mar65}  gave another proof of (\ref{e1}) and pointed out that (\ref{e1}) is equivalent to
  \begin{eqnarray}\label{e2z}
 \prod_{k=1}^n(1-r_k)\le |\det(I-A)| \le \prod_{k=1}^n(1+r_k)
 \end{eqnarray} for any  $n\times n$   matrix $A$ with the same singular values as the contractive matrix $Z$.

 Marcus's proof of (\ref{e2z}) makes use of majorization theory and singular value (eigenvalue)  inequalities of Weyl. This approach is still very fruitful today in deriving determinantal inequalities; see, for example, \cite{Lin15, Lin16}. At about the same time of Marcus's proof, Hua \cite{Hua65} gave a proof of    (\ref{e2z}) using the determinantal inequality
 he had previously obtained in \cite{Hua55}.
\newpage

\begin{rem}
{\rm We notice the following:
\begin{itemize}
\item[(i).]  In the well-known  book \cite{MOA11} of Marshall, Olkin and Arnold, Tung's theorem is cited as result E.3 on page 319 in which the condition that $A$ be contractive is missing; that is to say, Theorem \ref{thm1} need not be true in general if $Z$ is not contractive. Take, for example, $Z=2i I$ with  odd $n$ (the matrix size) and appropriate $U$.
     We see neither the left nor the right inequality in (\ref{e1})  holds. However,     the following inequality
    holds true for any $n\times n$ matrix $Z$ and any $n\times n$ unitary matrix $U$ (a fraction
  with  zero denominator is viewed as $\infty$)
    \begin{eqnarray*}
    \label{e1z} \prod_{k=1}^n\frac{|1-r_k|}{1+r_k}\le \frac{|\det(I-Z^*Z)|}{|\det(I-UZ)|^2}.
\end{eqnarray*}

    \item[(ii).]  Inequalities (\ref{e1}) and (\ref{e2z}) are not equivalent for general matrices. The right-hand side inequality in   (\ref{e2z})
     is true for all $n\times n$ matrices $A$; that is,
    \begin{eqnarray*}\label{e2}
  |\det(I-A)| \le \prod_{k=1}^n(1+r_k).
 \end{eqnarray*}
\end{itemize}
}
\end{rem}

Using the polar decomposition (see, e.g., \cite[p.~83]{Zha11}), we restate and slightly generalize Theorem~\ref{thm1}   as follows with discussions on the equality cases.

\begin{thm}\label{thm2} Let $Z$ be an $n\times n$ positive semidefinite matrix with eigenvalues $r_1, r_2, \dots, r_n$. Let $U$ be an $n\times n$ unitary matrix such that $I-UZ$ is nonsingular. \nolinebreak Then
\begin{eqnarray}\label{e3a}
\prod_{k=1}^n\frac{|1-r_k|}{1+r_k}\le \frac{|\det(I-Z^2)|}{|\det(I-UZ)|^2}
\end{eqnarray}
with equality if and only if $Z$ has an eigenvalue $1$ or $UZ$ has   eigenvalues
  $-r_1, -r_2,$ $ \dots, -r_n$. If both $Z$ and $I-Z$ are nonsingular,
   the strict inequality
 holds for $U\neq -I$.

Moreover, if   $0\le r_k<1$,   $k=1, 2, \ldots, n$, then
\begin{eqnarray}\label{e3b} \frac{\det(I-Z^2)}{|\det(I-UZ)|^2}\le \prod_{k=1}^n\frac{1+r_k}{1-r_k}
\end{eqnarray}
with equality if and only if $UZ$ has  eigenvalues
 $r_1, r_2, \dots, r_n$. If $Z$ is nonsingular, then the strict inequality  in (\ref{e3b}) holds if $U\neq I$.
\end{thm}

If we use ${\rm Spec}(X)$ to denote the spectrum of the matrix $X$, then
equality holds in (\ref{e3a}) if and only if $1\in {\rm Spec}(Z)$ or
${\rm Spec} (UZ)={\rm Spec} (-Z)$; and
equality holds in (\ref{e3b}) if and only if
${\rm Spec} (UZ)={\rm Spec} (Z)$.

To prove this theorem, we need a lemma which is of interest in its own right. We proceed with
adoption of standard notation in majorization theory (see, e.g., \cite{Zha11}). 

\begin{lem}
Let $x=(x_1, x_2, \dots, x_n)$ and   $y=(y_1, y_2, \dots, y_n) $ be nonnegative vectors and assume that
$y$ is not a permutation of $x$ (i.e., the multisets
$\{x_1, x_2, \dots, x_n\}$ and $\{y_1, y_2, \dots, y_n\}$ are not equal). Denote
$\tilde{z}=(1+z_1, 1+z_2, \dots, 1+z_n)$.  We have:
$$\mbox{If\quad  $x\prec_{\rm log} y,$ \quad then \quad
$\tilde{x} \prec_{\rm \rm wlog} \tilde{y}$ \;\; {but}\quad
$\tilde{x} \not \prec_{\rm log} \tilde{y}$}. $$
Consequently,
\begin{equation}\label{ineq1}
\prod_{k=1}^n(1+x_k)< \prod_{k=1}^n(1+y_k).
\end{equation}
\end{lem}


\begin{proof}  Without loss of generality, we assume that $x$ and $y$ are positive vectors. (Otherwise,
replace the 0's in $x$ and $y$ by sufficiently small positive numbers and use continuity argument.)   Since  $x\prec_{\rm log} y$, we have
$\ln x \prec\ln y$  (see, e.g., \cite[p.~344]{Zha11}).
Let $f(t)=\ln (1+e^t)$. Then $f$ is strictly
 increasing and convex.
By \cite[Theorem~10.12, p.~342]{Zha11}, we have
$f(\ln x)\prec_{\rm w}  f(\ln y)$; that is,
$\ln \tilde{x} \prec_{\rm w} \ln \tilde{y}$, i.e.,
$\tilde{x} \prec_{\rm wlog} \tilde{y}   $.
Since $x$ is not a permutation of $y$,
$\ln x$ is not a permutation of $\ln y$.
Applying  \cite[Theorem~10.14, p.~343]{Zha11}, we obtain
$\sum_{k=1}^nf(\ln x_k) <   \sum_{k=1}^nf(\ln y_k)$, namely, $\sum_{k=1}^n\ln (1+x_k) <
 \sum_{k=1}^n\ln (1+y_k)$  which yields
$\prod_{k=1}^n(1+x_k)< \prod_{k=1}^n(1+y_k). $ \end{proof}

Under the lemma's condition, if
all $x_i$ and $y_i$ are further  in $[0, 1)$, a similar but reversal log-majorization inequality can be derived for $(1-x_k)$ and $(1-y_k)$. In particular,
\begin{equation}\label{ineq2}
 \prod_{k=1}^n(1-x_k)> \prod_{k=1}^n(1-y_k).
\end{equation}

In fact, this can be proved by applying \cite[Theorem 10.14, p.~343]{Zha11}
to $\ln x \prec \ln y$ with  $g(t)=-\ln (1-e^t)$, which is strictly increasing and convex when $t\in (-\infty, 0)$.

\medskip

\noindent
{\bf \emph{Proof of Theorem 3.}}
We only need to show the equality cases. For (\ref{e3a}), if $Z$ has a singular (eigen-) value 1, then both sides vanish. If $UZ$ has eigenvalues $-r_1, -r_2, \dots, -r_n$, then  $\det (I-UZ)=\prod_{k=1}^n(1+r_k)$. Equality is readily seen.  Conversely, suppose equality occurs in (\ref{e3a}). We further assume that no $r_k$
($k=1, 2, \dots, n$) equals $1$.
Since  $|\det(I-Z^2)|=\prod_{k=1}^n|1-r_k|(1+r_k)$, we have
\begin{equation}\label{ze1}
|\det (I-UZ)|=\prod_{k=1}^n(1+r_k).
\end{equation}
Moreover, by Weyl majorization inequality (see, e.g.,
\cite[Corollary 10.2, p.~353]{Zha11}),
 $$|\lambda (UZ)| \prec_{\rm log} \sigma (UZ)=\sigma(Z)=\lambda  (Z),$$
where $\lambda (X)$ and $\sigma (X)$ denote the vectors of the eigenvalues and singular values of matrix $X$, respectively.
With $\lambda_k(X)$ denoting the eigenvalues of the $n\times n$ matrix $X$, $k=1, 2, \dots, n$, by the lemma,  we have
$$0<|\det (I-UZ)|=\prod_{k=1}^n|1-\lambda_k(UZ)|
\leq  \prod_{k=1}^n(1+|\lambda_k(UZ)|)\leq  \prod_{k=1}^n(1+r_k).$$
Thus, (\ref{ze1}) yields $|1-\lambda_k(UZ)|
=  1+|\lambda_k(UZ)|$ for all $k$, which implies $\lambda_k(UZ)\leq 0$ for $k=1, 2, \dots, n$, i.e., all eigenvalues of $-UZ$ are nonnegative.
 If     $|\lambda (UZ)|=\lambda (-UZ)$ is not a permutation of
$\lambda  (Z)$, then, by strict inequality (\ref{ineq1}), we have
$\prod_{k=1}^n(1+|\lambda_k(UZ)|)< \prod_{k=1}^n(1+\lambda_k(Z))=\prod_{k=1}^n(1+r_k)$, a contradiction
to (\ref{ze1}). It follows that $UZ$ has the eigenvalues $-r_1, -r_2, \dots, -r_n$.

For the equality in (\ref{e3b}), it occurs if and only if
$\prod_{k=1}^n(1-r_k)=|\det (I-UZ)|$. Note that $|\lambda (UZ)|\prec_{\rm log}\sigma (Z)$ and
\begin{equation}\label{ze2}
\prod_{k=1}^n|1-\lambda_k(UZ)|\geq \prod_{k=1}^n(1-|\lambda_k(UZ)|)\geq
\prod_{k=1}^n(1-\sigma_k(Z))=\prod_{k=1}^n(1-r_k).
\end{equation}
The first equality in (\ref{ze2})  occurs if and only if all
$\lambda_k(UZ)$ are in $[0, 1)$; the second equality occurs if and only if $\lambda (UZ)$ is a permutation of $\sigma (Z)$, i.e., ${\rm Spec} (UZ)={\rm Spec} (Z)$.

Now assume that $Z$ is nonsingular and suppose that equality holds
in (\ref{e3b}). Then $UZ$ has  eigenvalues
 $r_1, r_2, \dots, r_n$. Moreover, the singular values of
 $UZ$ are $r_1, r_2, \dots, r_n$. Let $P=UZ$. Then the eigenvalues of $P$ are just the singular values of $P$. So $P$ is positive definite.
   It follows that $U=PZ^{-1}$ has only positive eigenvalues. Since $U$ is unitary,
 $U$ has to be the identity matrix. The  case for
 (\ref{e3a}) is similar.
\hfill $\Box$
\medskip

In what follows,
we prove an extension of  Theorem \ref{thm1} that involves multiple  matrices. 

  \begin{thm}\label{thm3} Let $Z_i$, $i=1, 2, \ldots, m$,  be $n\times n$ positive semidefinite matrices. Suppose that the  eigenvalues of $Z_i$ are  $r_{ik}$   satisfying $0\le r_{ik}<1$,   $k=1, 2, \ldots, n$. Then for any $n\times n$ unitary matrix $U$ and positive scalars  $w_i$, $i=1, 2, \ldots, m$,   $\sum_{i=1}^mw_i=1$, we have
  \begin{eqnarray}\label{e4}
  \prod_{k=1}^n\prod_{i=1}^m\left(\frac{1-r_{ik}}{1+r_{ik}}\right)^{w_i}\le \frac{\det(I-(\sum_{i=1}^mw_i Z_i)^2)}{|\det(I-U\sum_{i=1}^mw_iZ_i)|^2}\le  \prod_{k=1}^n\prod_{i=1}^m
  \left(\frac{1+r_{ik}}{1-r_{ik}}\right)^{w_i}.
  \end{eqnarray}
Equality on the left-hand side occurs if and only if all $Z_i$ are equal to $Z$, say, and
$Z$ has an eigenvalue 1 or Spec($UZ$)=Spec($-Z$) (in which $U=-I$ if $Z$ is nonsingular); Equality on the right-hand side  occurs if and only if all $Z_i$ are equal to $Z$, say, and Spec($UZ$)=Spec($Z$) (in which $U=I$ if $Z$ is nonsingular).
  \end{thm}	
  	Clearly, if $m=1$, then (\ref{e4}) reduces to (\ref{e1}).

\begin{proof} We begin by noticing the fact that for $n\times n$ Hermitian matrices $A$ and $B$, if $\lambda_k(A+B)=\lambda_k (A)+\lambda_k(B)$ for all $k=1, 2, \dots, n$, where $\lambda_k(X)$ denotes the $k$th largest eigenvalue of $X$,  then
  $A$ and $B$ are simultaneously unitarily diagonalizable with their eigenvalues on the main diagonal in the same order. (The converse is also true.) This is easily seen from  \cite[Theorem 2.4]{So94}.
  The result holds for multiple matrices.

Our additional ingredients in proving (\ref{e4})   are Fan's majorization relation that $\lambda (H+S)\prec \lambda (H)+\lambda (S)$ for $n\times n$ Hermitian matrices $H$ and $S$ (see, e.g., \cite[p.~356]{Zha11})  and Lewent's inequality  \cite{JPS07, Lew08} which asserts, for $x_i\in [0, 1)$, $i=1, 2, \dots, n$,
$$
\frac{1+\sum_{i=1}^n \alpha_i x_i}{1-\sum_{i=1}^n \alpha_i x_i} \leq \prod_{i=1}^n \left ( \frac{1+x_i}{1-x_i}\right )^{\alpha_i},$$
where $\sum_{i=1}^n\alpha_i=1$ and all $\alpha_i>0$.  Equality holds if and only if $x_1=x_2=\cdots =x_n.$

Let $r_{ik}^\downarrow$ be the $k$th largest eigenvalue of $Z_i$, $i=1, 2,  \ldots, m$, and  $s_k$ be the $k$th largest eigenvalue of $W:=\sum_{i=1}^mw_iZ_i$, $k=1, 2,  \ldots, n$. Fan's majorization relation implies
\begin{eqnarray*}
\lambda(W)\prec \sum_{i=1}^mw_i\lambda(Z_i), \;\;
\mbox{i.e.}, \;\; \sum_{k=1}^\ell s_k\le \sum_{k=1}^\ell \sum_{i=1}^mw_ir_{ik}^\downarrow, \qquad \ell=1, 2,  \ldots, n.
\end{eqnarray*}
(Note that the components of $\lambda (\cdot)$ are in nonincreasing order.)
 Now the convexity and the monotonicity of the function $f(t)=\ln \frac{1+t}{1-t}$, $0\le t<1$, imply  (see, e.g., \cite[p. 343]{Zha11})
  \begin{eqnarray*}
  \sum_{k=1}^n \ln\frac{1+s_k}{1-s_k}\le  \sum_{k=1}^n \ln\frac{1+  \sum_{i=1}^mw_ir_{ik}^\downarrow}{1- \sum_{i=1}^mw_ir_{ik}^\downarrow},
  \end{eqnarray*}
 where  equality holds  if and only if
  $s_k=\sum_{i=1}^m w_ir_{ik}^\downarrow$ for all $k$;
 that is, $\lambda(W)=\sum_{i=1}^mw_i\lambda(Z_i)$.
It follows that all $Z_i$ are simultaneously unitarily diagonalizable with their eigenvalues on the main diagonals in the same order (nonincreasing, say).

   Applying the exponential function to both sides and using Lewent's inequality  yield  \begin{eqnarray}\label{p1} \begin{aligned}
    \prod_{k=1}^n \frac{1+s_k}{1-s_k}&\le&  \prod_{k=1}^n \frac{1+  \sum_{i=1}^m w_ir_{ik}^\downarrow}{1-  \sum_{i=1}^mw_ir_{ik}^\downarrow}  \\&\le&
    \prod_{k=1}^n \prod_{i=1}^m \left(\frac{1+ r_{ik}^\downarrow}{1-  r_{ik}^\downarrow}\right)^{w_i} \\&=&
        \prod_{k=1}^n \prod_{i=1}^m \left(\frac{1+ r_{ik}}{1-  r_{ik}}\right)^{w_i}, \end{aligned}
    \end{eqnarray} in which equality occurs in the second inequality if and only if
    $r_{1k}=r_{2k}=\cdots =r_{mk}$ for $k=1, 2, \dots, n$.  Thus both equalities in (\ref{p1})
    hold if and only if $Z_1=Z_2=\cdots =Z_m$.

    By  (\ref{e3b}), we have
    \begin{eqnarray}\label{p2}
         \frac{\det(I-(\sum_{i=1}^mw_i Z_i)^2)}{|\det(I-U\sum_{i=1}^mw_iZ_i)|^2}\le   \prod_{k=1}^n  \left(\frac{1+ s_k}{1-  s_k}\right).
        \end{eqnarray}
        Combining (\ref{p1}) and  (\ref{p2}) gives  the second inequality of (\ref{e4}).

       Note that the inequalities in (\ref{p1}) reverse by taking reciprocals, which implies
       \begin{eqnarray}\label{p3}
           \prod_{k=1}^n \frac{1-s_k}{1+s_k}\ge
               \prod_{k=1}^n \prod_{i=1}^m \left(\frac{1-r_{ik}}{1+r_{ik}}\right)^{w_i}.
           \end{eqnarray} Then by (\ref{e3a}), we have
               \begin{eqnarray}\label{p4}
                    \frac{\det(I-(\sum_{i=1}^mw_i Z_i)^2)}{|\det(I-U\sum_{i=1}^mw_iZ_i)|^2}\ge   \prod_{k=1}^n  \left(\frac{1- s_k}{1+  s_k}\right).
                   \end{eqnarray} Combining (\ref{p3}) and  (\ref{p4}) yields the first inequality of (\ref{e4}).

                   If either equality holds in (\ref{e4}), then
                   all $Z_i$ are equal to $Z$, say. The conclusions are immediate from Theorem \ref{thm2}.
\end{proof}
 	
 	\begin{rem}
 {\rm
 	Interestingly, it is observed in \cite{JPS07} that the Lewent's inequality follows directly from the  convexity   of $\displaystyle{f(t)=\ln\frac{1+t}{1-t}}$, \  $0\le t<1$, applied to the Jensen's inequality. Note that equality occurs in the Lewent's inequality if and only if all the variables $x_1, x_2, \dots, x_n$ are identical
 (see, e.g., \cite[Theorem 2, p.~3]{Boll99}).
 }
 	\end{rem}

 The absolute value of a complex   matrix $Z$, denoted by $|Z|$,   is the positive semidefinite square root of $Z^*Z$.	
 The following result  extends the second inequality in (\ref{e1}).

 	\begin{cor} Let $Z_i$, $i=1,  2, \ldots, m$,  be $n\times n$ complex matrices with   singular values   $r_{ik}$  such  that $0\le r_{ik}<1$, $k=1, 2, \ldots, n$.   Then for any $n\times n$ unitary matrix $U$
 	  \begin{eqnarray*}
 	   \frac{\det(I-\sum_{i=1}^mw_i Z_i^*Z_i)}{|\det(I-U\sum_{i=1}^mw_i|Z_i|)|^2}\le  \prod_{k=1}^n\prod_{i=1}^m\left(\frac{1+r_{ik}}{1-r_{ik}}\right)^{w_i},
 	  \end{eqnarray*} where $w_i>0$, $i=1, 2, \ldots, m$, such that $\sum_{i=1}^mw_i=1$. Equality occurs if and only if all $Z_i$ have the same absolute value, say $Z$, and Spec($UZ$)=Spec($Z$) (in which $U=I$ if $Z$ is nonsingular).
 	\end{cor}

 \begin{proof} With $(\sum_{i=1}^m w_i|Z_i|)^2\leq \sum_{i=1}^mw_i|Z_i|^2$
 (see, e.g., \cite[p.~5]{Zha02}), use Theorem \ref{thm3}.
 \end{proof}

In view of (\ref{e4}), it is tempting to have the  lower bound inequality
\begin{eqnarray*}
 \prod_{k=1}^n\prod_{i=1}^m\left(\frac{1-r_{ik}}{1+r_{ik}}\right)^{w_i}\le    \frac{\det(I-\sum_{i=1}^mw_i Z_i^*Z_i)}{|\det(I-U\sum_{i=1}^mw_i|Z_i|)|^2}.
 	  \end{eqnarray*}
However, this is not true. Set $m=n=2$, $w_1=w_2=1/2$ and take
$$Z_1 = \left ( \begin{array}{cc}
 0.34  &   -0.15 \\
   -0.15  &   0.07 \end{array} \right ),\;\;
Z_2 =\left ( \begin{array}{cc}
    0.02  &  -0.01 \\
   -0.01  &   0.01 \end{array} \right ),\;\;
   U =
\left ( \begin{array}{cc}
   -0.60  &  0.80 \\
    0.80   &  0.60
    \end{array} \right ).$$
  One may check that the left hand side is
  $0.6281$,
while the right hand side is
$0.6250$.

 The unitary matrix in Theorem \ref{thm1} seems superfluous, as one could replace $Z$ with $U^*Z$ and leave the singular values unchanged. So perhaps a more natural extension of Theorem~\ref{thm1} to more matrices is giving an upper bound and lower bound, in terms of the singular values of individual matrices, for the quantity ${\displaystyle\frac{\det(I-\sum_{i=1}^mw_i Z_i^*Z_i)}{|\det(I-\sum_{i=1}^mw_i Z_i)|^2}}$, where $Z_i$, $i=1, 2, \ldots, m$, are general contractive matrices.
 We would guess
\begin{equation}\label{open}
\prod_{k=1}^n\prod_{i=1}^m\left(\frac{1-r_{ik}}{1+r_{ik}}\right)^{w_i}\le \frac{\det(I- \sum_{i=1}^mw_i Z_i^*Z_i)}{|\det(I- \sum_{i=1}^mw_iZ_i)|^2}\le  \prod_{k=1}^n\prod_{i=1}^m\left(\frac{1+r_{ik}}{1-r_{ik}}\right)^{w_i}.
\end{equation}
The first inequality in (\ref{open})  is untrue in general as it is disproved by substituting $Z_1$ and $Z_2$ in (\ref{open}) with $U|Z_1|$ and $U|Z_2|$, respectively,  in the previous example.
However, simulation seems to support the second inequality which is   unconfirmed yet.

\subsection*{Acknowledgments} {\small   The work
 was partially supported by National Natural Science Foundation of China (NNSFC) No.~11601314 and No.~11571220.}

\end{document}